\documentclass[11pt,reqno]{amsart}

\usepackage{amsthm,amsmath,amssymb}
\usepackage{graphicx}
\usepackage{float}
\usepackage[colorlinks=true,
linkcolor=blue,
urlcolor=blue,
citecolor=blue]{hyperref}
\usepackage{mathtools}
\usepackage{relsize}
\usepackage{tikz}
\usepackage{enumerate}
\usepackage[toc,page]{appendix}
\usepackage{pdflscape}
\usepackage{multirow}
\usepackage{adjustbox,expl3,etoolbox} 
\usepackage[margin = 3cm]{geometry}
\usepackage{mathrsfs}
\usepackage{rotating}
\usepackage{tablefootnote}
\usepackage{longtable}
\usepackage{comment}
\usepackage{xcolor}
\usepackage{stmaryrd}
\usepackage{url}

\newtheorem{thm}{Theorem}[section]
\newtheorem{lem}[thm]{Lemma}
\theoremstyle{definition}
\newtheorem{defn}[thm]{Definition}
\newtheorem{prop}[thm]{Proposition}
\newtheorem{conj}[thm]{Conjecture}
\newtheorem{cor}[thm]{Corollary}

\newtheorem{rmk}[thm]{Remark}

\newtheorem{qst}[thm]{Question}

\newcommand{\F}{\mathbb{F}}
\newcommand{\Z}{\mathbb{Z}}

\newcommand{\T}{\overline{T}}

\title{The Collatz map analogue in polynomial rings and in completions}
\author{Angelot Behajaina}
\email{abeha@campus.technion.ac.il}
\address{Department of Mathematics, Technion - Israel Institute of Technology, Haifa, Israel}
\address{Department of Mathematics and Computer Science, The Open University of Israel}

\author{Elad Paran}
\email{paran@openu.ac.il}
\address{Department of Mathematics and Computer Science, The Open University of Israel}

\begin{document}

\maketitle
\begin{abstract}

We study an analogue of the Collatz map in the polynomial ring $R[x]$, where $R$ is an arbitrary commutative ring. We prove that if $R$ is of positive characteristic, then every polynomial in $R[x]$ is eventually periodic with respect to this map. This extends previous works of the authors and of Hicks, Mullen, Yucas and Zavislak, who studied the Collatz map on $\F_p[x]$ and $\F_2[x]$, respectively. We also consider the Collatz map on the ring of formal power series $R[[x]]$ when $R$ is finite: we characterize the eventually periodic series in this ring, and give formulas for the number of cycles induced by the Collatz map, of any given length. We provide similar formulas for the original Collatz map defined on the ring $\Z_2$ of $2$-adic integers, extending previous results of Lagarias. 

\end{abstract}
\section{Introduction}

The classical Collatz map is defined for any natural number $n$ by $T(n) = n/2$ if $n$ is even and $T(n) = 3n+1$ if $n$ is odd. The Collatz conjecture states that for any $n \geq 1$ we have $\big(T^k(n),T^{k+1}(n),\ldots \big) = \big(1,4,2,1,4,2,\ldots\big)$ for some $k \geq 1$. While some progress has been made in recent years by Tao \cite{Tao2022}, the conjecture remains wide open. For a survey on the history of this question, see \cite{Lagarias2010}.

In the recent work \cite{BP23}, inspired by the work \cite{Hicks} of Hicks, Mullen, Yucas and Zavislak, the authors consider an analogue of the Collatz map, defined on the ring $\F_p[x]$ of polynomials over a prime field $\F_p$, by:

$$
T(f)=\begin{cases}
f \cdot (x+1)-f_0 &\textrm{if}\,\, f_0 \neq 0,\\
\frac{f}{x} & \textrm{otherwise}.
\end{cases}
$$

(Here, $f_0$ denotes the constant coefficient of  $f$.) One says that a polynomial $f$ is {\bf periodic} (with respect to $T$) if there exists $k \geq 1$ such that $T^k(f) = f$. If $f$ is periodic and such a $k$ is minimal, then we call the sequence $(f,T(f),T^2(f),\ldots,T^{k-1}(f))$ a {\bf $T$-cycle}. One says that $f$ is {\bf eventually periodic} or {\bf pre-periodic} (with respect to $T$), if there exists $m \geq 1$ such that $T^m(f)$ is periodic. 

The main result of \cite{BP23} is that every polynomial $f \in \F_p[x]$ is eventually periodic with respect to $T$, in a quadratic number in $\deg(f)$ of $T$-iterations, and precise description of all $T$-cycles is given, see \cite[Theorem 1.1]{BP23}. This extends the main result of \cite{Hicks}, where the case of $\F_2$ is analyzed. In \cite{BP23} the authors also considered the extension of $T$ to the formal power series ring $\F_p[[x]]$, and determined the eventually periodic elements in this ring with respect to $T$, see \cite[Theorem 1.2]{BP23}.

The goal of the present work is to study the Collatz map in $R[x]$ and $R[[x]]$, given by the same definition as above, where $R$ is any commutative ring (possibly infinite). In the sequel, we assume $R$ to be a {\bf nonzero ring}. Concerning the polynomial ring $R[x]$, our main result is the following:

\begin{thm}\label{main:thm1}
Let $R$ be a commutative ring of characteristic $\mathfrak{N}$.

\begin{enumerate}
\item If $\mathfrak{N}=0$, then the only  $T$-cycles in $R[x]$ are $(0)$ and those of the form $( a,ax)$, for all $0 \neq a \in R$.
\item Assume $\mathfrak{N}=\prod_{i=1}^sp_i^{\alpha_i}>0$, where $s \geq 1$, $p_1,\dots,p_s$ are distinct primes, and $\alpha_1,\dots,\alpha_s \geq 1$. Then:
\begin{itemize}
\item Every $f \in R[x] $ is eventually periodic with respect to $T$. 
\item Let $f=\sum_{k=0}^{n}b_kx^k \in R[x]$ be a periodic polynomial with $b_0, b_n \neq 0$ and let $\rho_T(f)$ denote the length of the corresponding $T$-cycle. Then  
\begin{equation}\label{eq:periodivmain}
\rho_T(f) \mid 2 \prod_{i=1}^s p_i^{\alpha_i+\lfloor \log_{p_i}(\deg(f)) \rfloor}.
\end{equation} 
Moreover if $b_{n}$ is in $R^\times$, then we have equality in \eqref{eq:periodivmain}:
\begin{equation*}
\rho_T(f) = 2 \prod_{i=1}^s p_i^{\alpha_i+\lfloor \log_{p_i}(\deg(f)) \rfloor}.
\end{equation*}

\end{itemize}
\end{enumerate}
\end{thm}

We note that in the case where $R$ is finite, the fact that every polynomial in $R[x]$ is  eventually periodic is immediate, see our remark following Lemma \ref{lem:basiresT} below. The case where $R$ is an infinite ring of positive characteristic is more complicated and is the main novelty of Theorem \ref{main:thm1}.

In \cite{Lag90}, Lagarias considered the extension of the classical Collatz map to the ring $\Z_2$ of $2$-adic integers.  He  proved that under the ``condensed" Collatz map:
\begin{align*}
\overline{T} : \mathbb{Z}_2 & \rightarrow \mathbb{Z}_2\\
           f &\mapsto \begin{cases}
    \frac{f}{2} & \textrm{if}\,\, f \equiv 0 \pmod{2},\\
    \frac{3f+1}{2} & \textrm{otherwise},
           \end{cases}
\end{align*} the number of cycles of length $n \geq 1$ is 
$$
I(n)=\frac{1}{n}\sum_{d|n} \mu(d)2^{n/d},
$$ where $\mu$ denotes the M\"{o}bius function. As a consequence, he proved that $I(n) \sim \frac{2^n}{n}$ as $n \rightarrow \infty$. The ring $\Z_2$ is the $2$-adic completion of $\Z$, and its formal analogue is the ring $R[[x]]$ of formal power series (the $x$-adic completion of $R[x]$). For $R[[x]]$ with $R$ a finite ring, inspired by the method used in \cite{Lag90}, we characterize all eventually $T$-periodic elements, prove that there are finitely many cycles of any possible length for the map $T$ itself, and provide a formula for the number of such cycles: 

\begin{thm}\label{main:thm2}
Let $R$ be any {\bf finite} ring of cardinality $q$.
\begin{enumerate}
\item The eventually $T$-periodic power series in $R[[x]]$ are precisely series of the form $u(1+xv)^{-1}$, where $u,v$ are polynomials in $R[x]$.
\item Let $n \geq 1$. Then the number $Z_n$ of $T$-cycles of length $n$ in $R[[x]]$ is finite. Moreover, letting $b=4q-3$, $\alpha=\frac{1+\sqrt{b}}{2}$ and $\beta=\frac{1-\sqrt{b}}{2}$, we have 
\begin{equation*}
Z_n=\frac{1}{n}\sum_{d|n}\mu(d)\left(\alpha^{\frac{n}{d}}+\beta^{\frac{n}{d}}\right).
\end{equation*} Asymptotically, we have
$$
Z_n \sim \frac{\alpha^n}{n},
$$ as $n \rightarrow  \infty$.
\end{enumerate}
\end{thm}
\begin{rmk}
The first part of Theorem \ref{main:thm2} cannot be extended to the case where $R$ has characteristic $0$. In fact, there are polynomials in $R[x]$ that are not eventually $T$-periodic. For example, when $R=\mathbb{Z}$, we have $T^{2k}(x+1)=x+k+1$ for all $k \geq 0$; hence $x+1$ is not eventually $T$-periodic by the first part of Theorem \ref{main:thm1}.
\end{rmk}

We note that the second part of Theorem \ref{main:thm2} is novel even in the case where $R = \F_p$ is a prime field -- 
this question was not studied in \cite{BP23}.

Lagarias's mentioned result on the number of cycles is given for the condensed Collatz map $U$. For the original Collatz map, the authors could not find available formulas in the literature. Therefore, to complete the picture and see the analogy between $R[[x]]$ and $\Z_2$, we compute here the number of cycles for the usual Collatz map $T$ on $\Z_2$, of any given length, and obtain the following result:

\begin{thm}\label{main_z2}
Let $n \geq 1$. Then the number $Z_n$ of $T$-cycles of length $n$ in $\mathbb{Z}_2$ is finite, given by:
\begin{align*}
Z_n=\frac{1}{n}\sum_{d|n}\mu(d) \left( \phi^{\frac{n}{d}}+\psi^{\frac{n}{d}}\right),
\end{align*}
where $\phi=\frac{1+\sqrt{5}}{2}$ and $\psi=\frac{1-\sqrt{5}}{2}$.
Asymptotically, we have
$$
Z_n \sim \frac{\phi^n}{n},
$$ as $n \rightarrow  \infty$.\end{thm}

Our proof strategy for Theorem \ref{main_z2} is very similar to the strategy we employ for $R[[x]]$. 

This paper is organized as follows: In section \ref{sec:poly}, we study the Collatz map on $R[x]$ and prove Theorem \ref{main:thm1}. In Section \ref{sec:power} we study the Collatz map on $R[[x]]$ and prove Theorem \ref{main:thm2}. Finally, in section \ref{sec:z2} we prove Theorem \ref{main_z2}.

\section{The Collatz map in polynomial rings}\label{sec:poly}

In this section, we prove Theorem \ref{main:thm1}.
\subsection{Basic results on valuations} For any prime $p$:

\begin{itemize}
\item We denote by $v_p$ the corresponding $p$-adic valuation. 
\item For any non-negative integer $n=n_0+n_1p+\dots+n_rp^r$ ($r \geq 0$ and $0 \leq n_i \leq p-1$ for $0 \leq i \leq r$), written in base $p$, let denote $S_p(n)=n_0+ \dots + n_r$. 
\end{itemize}
Recall the following well-known result (see \cite{Rowland2018}): 
\begin{thm}[Kummer's theorem]\label{thm:kum}
Let $p$ be a prime. Then we have
$$
v_p\left(  \binom{n}{m} \right)=\frac{S_p(m)+S_p(n-m)-S_p(n)}{p-1},
$$ for all integers $0 \leq m \leq n$.
\end{thm}
In the sequel, for any $a, b \in \mathbb{Z}$, we let
$$
{\bf 1}_{a \geq b}=
\begin{cases}
1 & \textrm{if}\,\, a\geq b,\\
0 & \textrm{otherwise}.
\end{cases}
$$
\begin{lem}\label{lem:valbinote}
Let $p$ be a prime. Let $a \geq n \geq 1$ such that $v_p(a) \geq  \lfloor \log_p(n) \rfloor +1$. Then
\begin{align*}\label{eq:formvalu}
v_p\left( \binom{a}{n}\right)=v_p(a)-v_p(n).
\end{align*}
\end{lem}
\begin{proof}
 Write $n=n_{k_0}p^{k_0}+n_{k_1}p^{k_1}+\dots+n_{k_v}p^{k_v}$, where $v\geq 0$, $0 \leq k_0<k_1 < \dots < k_v$ and $ 1 \leq n_{k_0},\dots,n_{k_v} \leq p-1$. Write $a=a_{\ell_0}p^{\ell_0}+a_{\ell_1}p^{\ell_1}+\dots+a_{\ell_w}p^{\ell_w}$, where $w\geq 0$, $0 \leq \ell_0<\ell_1 < \dots < \ell_w$ and $ 1 \leq a_{\ell_0},\dots,a_{\ell_w} \leq p-1$. Note that $v_p(a) \geq  \lfloor \log_p(n) \rfloor +1$ is the same as $\ell_0 \geq k_v +1$. Then we have
\begin{align*}
a-n&= \left(a_{\ell_w}p^{\ell_w}+\dots+a_{\ell_0}p^{\ell_0}\right)-\left(n_{k_v}p^{k_v}+n_{k_{v-1}}p^{k_{v-1}}+\dots+n_{k_0}p^{k_0}\right)\\
&= \left(a_{\ell_w}p^{\ell_w}+\dots+a_{\ell_1}p^{\ell_1}+(a_{\ell_0}-1)p^{\ell_0}\right)+\left(p^{\ell_0}-n_{k_v}p^{k_v}-n_{k_{v-1}}p^{k_{v-1}}-\dots-n_{k_0}p^{k_0}\right)\\
&= \left(a_{\ell_w}p^{\ell_w}+\dots+a_{\ell_1}p^{\ell_1}+(a_{\ell_0}-1)p^{\ell_0}\right)\\
&+ {\bf 1}_{\ell_0 \geq k_v+2}\sum_{i=k_v+1}^{\ell_0-1}(p-1)p^i+\sum_{i=1}^v(p-n_{k_i}-1)p^{k_{i}}\\
&+\sum_{j=0}^{v-1} {\bf 1}_{k_{j+1} \geq k_j+2}\sum_{i=k_j+1}^{k_{j+1}-1}(p-1)p^i+(p-n_{k_0})p^{k_0}.
\end{align*}
A straightforward computation gives $S_p(a-n)=S_p(a)+(p-1)(\ell_0-k_0)-S_p(n)$. From Theorem \ref{thm:kum}, we get
$$
v_p\left( \binom{a}{n}\right)=\frac{S_p(n)+S_p(a-n)-S_p(a)}{p-1}=\ell_0-k_0=v_p(a)-v_p(n),
$$ as was to be shown.
\end{proof}
\begin{lem}\label{lem:calded} Let $p$ be a prime.
Let $\alpha \geq 1$, $r \geq 0$ and $k \geq p^r$.  Assume that
$$
v_p\left( \binom{k}{j} \right) \geq \alpha,
$$ for all $1 \leq j \leq p^r$. Then we have $v_p(k) \geq \alpha +r$.
\end{lem}
\begin{proof}
Write $k=k_{\ell_0}p^{\ell_0}+ \dots +k_{\ell_v}p^{\ell_v}$, where $v \geq 0$, $0 \leq \ell_0 < \dots < \ell_v$ and $1 \leq k_{\ell_0},\dots, k_{\ell_v} \leq p-1$. By contradiction, assume that $r > \ell_0$. Since $k \geq p^r$, we have $r \leq \ell_v$. Hence there exists $0 \leq i \leq v-1$ such that $\ell_i <r \leq \ell_{i+1}$. From $1 \leq \sum_{j=0}^{i}k_{\ell_j}p^{\ell_j} \leq p^r$, we get
$$
\alpha \leq v_p\left( \binom{k}{\sum_{j=0}^{i}k_{\ell_j}p^{\ell_j}} \right)=\frac{S_p\left(\sum_{j=0}^{i}k_{\ell_j}p^{\ell_j}\right)+S_{p}\left(k-\sum_{j=0}^{i}k_{\ell_j}p^{\ell_j}\right)-S_{p}(k)}{p-1}=0,
$$ a contradiction. Consequently we have $r \leq \ell_0$. From the equality
\begin{align*}
k-p^r&= k_{\ell_v}p^{\ell_v}+ \dots+ k_{\ell_0}p^{\ell_0}-p^r\\
&= k_{\ell_v}p^{\ell_v}+ \dots + k_{\ell_1}p^{\ell_1}+(k_{\ell_0}-1)p^{\ell_0}+{\bf 1}_{\ell_0\geq r+1}\sum_{j=r}^{\ell_0-1}(p-1)p^j,
\end{align*} we obtain
$$
S_p(k-p^r)=S_{p}(k)-1+(p-1)(\ell_0-r),
$$ and hence
$$
\alpha \leq v_p\left( \binom{k}{p^r} \right)=\frac{S_p(p^r)+S_p(k-p^r)-S_p(k)}{p-1}=\ell_0-r=v_p(k)-r.
$$ Therefore we get $v_p(k) \geq \alpha+r$.
\end{proof}
\begin{lem}\label{lem:iffdivcond} Let $p_1,\dots,p_s$ ($s \geq 1$) be distinct primes and let $ \alpha_1, \dots ,\alpha_s \geq 1$ be integers. Let $k \geq n \geq 1$ be integers. The following are equivalent:
\begin{enumerate}
\item The product $\prod_{i=1}^s p_i^{\alpha_i+\lfloor \log_{p_i}(n) \rfloor}$ divides $k$.\\
\item For all $1 \leq i \leq s$ and all $1 \leq j \leq n$, we have
$$
v_{p_i}\left( \binom{k}{j} \right) \geq \alpha_i.
$$ In other words, we have
$$
\prod_{i=1}^s p_i^{\alpha_i} \mid \binom{k}{j},
$$ for all $1 \leq j \leq n$.
\end{enumerate}
\end{lem}
\begin{proof}
Assume that (1) holds, that is, $\prod_{i=1}^s p_i^{\alpha_i+\lfloor \log_{p_i}(n) \rfloor} \mid k$. Let $1 \leq i \leq s$ and $1 \leq j \leq n$. Since $v_{p_i}(k) \geq \alpha_i+ \lfloor \log_{p_i}(n) \rfloor \geq 1+ \lfloor \log_{p_i}(j) \rfloor$, by Lemma \ref{lem:valbinote}, we have
$$
v_{p_i}\left( \binom{k}{j}\right)= v_{p_i}(k)-v_{p_i}(j) \geq v_{p_i}(k)-\lfloor \log_{p_i}(j)\rfloor \geq v_{p_i}(k)-\lfloor \log_{p_i}(n)\rfloor \geq \alpha_i.
$$ Therefore we obtain (2).

Conversely assume that (2) holds, that is,
$$
v_{p_i}\left( \binom{k}{j} \right) \geq \alpha_i,
$$ for all $1 \leq i \leq s$ and all $1 \leq j \leq n$. Now let $1 \leq i \leq s$. Taking $p=p_i$, $r=\lfloor \log_{p_i}(n)\rfloor$ and $\alpha=\alpha_i$ in Lemma \ref{lem:calded}, and noticing that $p^r \leq n$, we deduce
$$
v_{p_i}(k) \geq \alpha_i+ \lfloor \log_{p_i}(n)\rfloor.
$$ Consequently, we obtain $\prod_{i=1}^s p_i^{\alpha_i+\lfloor \log_{p_i}(n) \rfloor} \mid k$. Hence (1) holds.
\end{proof}
\subsection{Main results}\label{subsec:Rx} Let $R$ be any commutative ring of characteristic $\mathfrak{N}$. Denote by $R^\times$ its group of units. We define the {\it Collatz map} by:
\begin{align*}
T : R[x] & \rightarrow R[x]\\
f & \mapsto
\begin{cases}
f \cdot (x+1)-f_0 & \textrm{if}\,\, f_0 \neq 0,\\
\frac{f}{x} & \textrm{otherwise}.
\end{cases}
\end{align*}
\begin{defn}\label{def:somedef}
\begin{enumerate}
\item
Let $f \in R[x]$. We say that $f$ is :
\begin{itemize}
    \item $T$-{\it periodic} if there exists $n \geq 1$ such that $T^{n}(f)=f$. The minimal such $n$ is called the {\it period} of $f$, and is denoted by $\rho_T(f)$.
    \item {\it eventually $T$- periodic} if there exist $m \geq 0$ and $n \geq 1$ such that $T^{n+m}(f)=T^m(f)$; that is, if some $T$-iterate of $f$ is $T$-periodic. 
\end{itemize}
\item
A \emph{$T$-cycle} $\mathcal{C} \subset R[x]$ is the orbit of some $T$-periodic polynomial $f \in R[x] $. The cardinality $\ell(\mathcal{C})$ of $\mathcal{C}$ is called the \emph{length} of $\mathcal{C}$. In this case, we have $\ell(\mathcal{C})=\rho_T(f)$.
\item
A polynomial $f \in R[x]$ is {\it odd} (resp., {\it even}) if $f(0) \neq 0$ (resp., $f(0)=0$).
\end{enumerate}
\end{defn}
We need the following two results, also established in \cite{BP23} when $R=\F_p$ ($p$ prime); however, we omit the proofs since they are similar to the proofs in \cite{BP23}.
\begin{lem}\label{lem:basiresT}
Let $f \in R[x] \setminus \{0\}$. The following hold:
\begin{enumerate}
\item We have $\deg(T(f)) \leq \deg(f)+1$. 
\item If $\deg(T(f))>\deg(f)$, then $f$ is odd and $T(f)$ is even.
\item We have $\deg(T^2(f)) \leq \deg(f)$.
\item For all $j \geq 0$, we have $\deg(T^{2j+1}(f)) \leq \deg(f)+1$ and $\deg(T^{2j}(f)) \leq \deg(f)$.
\end{enumerate}
\end{lem}
\begin{proof}
See \cite[Lemma 2.7.]{BP23}.
\end{proof}

Note that from part $4$ of Lemma \ref{lem:basiresT} it immediately follows that if $R$ is finite, then every polynomial $f \in R[x]$ is eventually $T$-periodic, since in this case the space of polynomials of degree at most $\deg(f)+1$ is finite.

\begin{lem}\label{cyc:oddeve}
Let $\mathcal{C} \subset R[x] \setminus \{0\}$ be a $T$-cycle. Then $f \in \mathcal{C}$ is even if and only if $T(f)$ is odd. In particular, the length $\ell(\mathcal{C})$ is even and $\mathcal{C}$ contains the same number of odd polynomials and even polynomials.
\end{lem}
\begin{proof}
See \cite[Lemma 2.9]{BP23}.
\end{proof}

In the sequel, we need the following intermediate map, also introduced in \cite[Section 2]{BP23}:
\begin{align*}
\Pi: R[x] & \rightarrow R[x]\\
     f & \mapsto \frac{f \cdot (x+1)-f_0}{x}.
\end{align*}
We recall some of the properties of $\Pi$ (see \cite[\S2]{BP23}): Let $n \geq 0$. Denote by $R[x]_{\leq n}$ the additive group of all polynomials of degree at most $n$. Clearly, we have $\Pi(R[x]_{\leq n}) \subset R[x]_{\leq n}$. Let $\Pi_n$ be the restriction of $\Pi$ to $R[x]_{\leq n}$. We define the $R$-linear shift map
\begin{equation*}
L_n: (b_n,\dots,b_0) \in R^{n+1} \mapsto (0,b_n,\dots,b_1) \in R^{n+1}.
\end{equation*}
By identifying $R^{n+1}$ with $R[x]_{\leq n}$ via $(b_n,\dots,b_0) \mapsto \sum_{i=0}^n b_ix^i$, note that $\Pi_n={\rm Id}_{R[x]_{\leq n}}+L_n$.
By $R$-linearity, for all $k \geq 0$, we have
\begin{equation}\label{eq:formpibin}
\Pi^k_n(f)=\sum_{j=0}^k \binom{k}{j} L^j_n(f),
\end{equation}
for all $f \in R[x]_{\leq n}$.
\subsubsection{Zero characteristic} Assume $\mathfrak{N}=0$. In contrast to positive characteristic, there are only few $T$-cycles in $R[x]$:
\begin{prop}\label{prop:zero_char}
The only $T$-cycles in $R[x] \setminus \{0\}$ are those of the form $(a,ax)$ (for any $a \in R \setminus\{0\}$).
\end{prop}
\begin{proof} It is easy to see that every $(a,ax )$ ($a \neq 0$) is a $T$-cycle. Conversely, let $\mathcal{C} \subset R[x]\setminus \{0\}$ be a $T$-cycle. By Lemma \ref{cyc:oddeve}, there exists an odd polynomial $f \in \mathcal{C}$. Let $d=\deg(f)$. Then for each $j\geq 0$, the polynomial $T^{2j}(f)$ is odd of degree $d$ and $T^{2j+1}(f)$ is even of degree $d+1$ (the same argument as in the beginning of the proof of \cite[Lemma 2.15.]{BP23}). By Lemma \ref{cyc:oddeve}, we have $\ell(\mathcal{C})=2m$ for some $m \geq 1$. In particular, we have $f=T^{2m}(f)=\Pi_d^m(f)=\sum_{j=0}^m \binom{m}{j}L_d^j(f)$ (see \eqref{eq:formpibin}), and so $0=\sum_{j=1}^m \binom{m}{j}L_d^j(f)$. If $d>0$, then comparing the first coordinate in this equation, the coefficient of $x^d$ in $f$ is zero, a contradiction. Thus $d=0$ and $f$ is a nonzero constant. Therefore we obtain $\mathcal{C}=(f_0,f_0x)$.
\end{proof}
\subsubsection{Positive characteristic} Assume $\mathfrak{N}>0$. Write $\mathfrak{N}=\prod_{i=1}^s p_i^{\alpha_i}$, where $p_1,\dots,p_s$ are distinct primes and $n_1,\dots,n_s \geq 1$. In this part, $\rho_\Pi(f)$ denotes the period of any periodic polynomial $f \in R[x]$ under the map $\Pi$.
The following result generalizes \cite[Corollary 2.13.]{BP23}:
\begin{prop}\label{prop:periodze}
Let $f=\sum_{k=0}^{n}b_kx^k \in R[x]$ be a $T$-periodic odd polynomial with $b_n \neq 0$. Then we have 
\begin{equation}\label{eq:periodiv}
\rho_T(f) \mid 2 \prod_{i=1}^s p_i^{\alpha_i+\lfloor \log_{p_i}(\deg(f)) \rfloor}.
\end{equation} 
Moreover if $b_{n}$ is in $R^\times$, then \eqref{eq:periodiv} is an equality, that is,
\begin{equation*}
\rho_T(f) = 2 \prod_{i=1}^s p_i^{\alpha_i+\lfloor \log_{p_i}(\deg(f)) \rfloor}.
\end{equation*}
\end{prop}
\begin{proof}
By \eqref{eq:formpibin}, for all $k \geq 1$, we have
\begin{equation}\label{eq:piform}
\Pi^k_n(f)=\sum_{j=0}^k \binom{k}{j} L^j_n(f).
\end{equation}
In particular, for $k=\prod_{i=1}^s p_i^{\alpha_i+\lfloor \log_{p_i}(n) \rfloor}$, by Lemma \ref{lem:iffdivcond}, we obtain
$\Pi_n^k(f)=f$; hence, we have $\rho_{\Pi}(f) |k $. Using Lemma \ref{cyc:oddeve}, we get $\rho_T(f)=2\rho_{\Pi}(f)$, and so we obtain \eqref{eq:periodiv}. 

Now assume that $b_{n}$ is in $R^\times$. Let $k=\rho_{\Pi}(f)=\dfrac{\rho_T(f)}{2}$. First, we claim that $k > n$. By contradiction, assume that $k \leq n$. From $\Pi^{k}_n(f)=f$, using \eqref{eq:piform}, we deduce $\sum_{j=1}^{k} \binom{k}{j}L_n^j(f)=0$, and in particular,
\begin{equation*}
\begin{cases}
\binom{k}{1}b_n =0,\\
\binom{k}{1}b_{n-1}+\binom{k}{2}b_n=0,\\
\dots\\
\binom{k}{1}b_{n-k+1}+\dots+\binom{k}{k}b_n=0.
\end{cases}
\end{equation*}
Since $b_n$ is in $R^{\times}$, we get 
$$
\binom{k}{1}=\dots=\binom{k}{k}\equiv 0 \pmod{\mathfrak{N}},
$$ and so $1=\binom{k}{k}\equiv 0 \pmod{\mathfrak{N}}$, a contradiction. Consequently, we have $k > n$. Using the same argument as before, we also have
$$
\binom{k}{1}= \dots =\binom{k}{n} \equiv 0 \pmod{\mathfrak{N}}.
$$
By Lemma \ref{lem:iffdivcond}, $\prod_{i=1}^s p_i^{\alpha_i+\lfloor \log_{p_i}(n) \rfloor}$ divides $k$. Consequently, we obtain 
$$\rho_T(f)=2\prod_{i=1}^s p_i^{\alpha_i+\lfloor \log_{p_i}(n) \rfloor}.$$ This completes the proof of the last statement of the proposition.
\end{proof}
\begin{lem}\label{crit:periodicity}
Let $f=\sum_{k=0}^{n}b_k x^k \in R[x]$ be an odd polynomial with $b_n \neq 0$. Then $f$ is $T$-periodic if and only if
\begin{align*}\label{eq:condper}
\sum_{j=0}^{\ell}\binom{\ell}{j}b_j \neq 0,
\end{align*}
for all $0 \leq \ell \leq \prod_{i=1}^s p_i^{\alpha_i+\lfloor \log_{p_i}(n) \rfloor}$.
\end{lem}
\begin{proof}[Proof]
We sketch a proof which is similar to the one of \cite[Lemma 2.15]{BP23}. Assume that $f$ is $T$-periodic. Then, by degree argument, $T^{2\ell}(f)$ (resp., $T^{2\ell+1}(f)$) must be  odd (resp., even) for all $\ell \geq 0$. This is equivalent to the fact that $\sum_{j=0}^{\ell}\binom{\ell}{j}b_j$ is the constant term of $T^{2\ell}(f)$ and is nonzero, for all $\ell \geq 0$. Therefore $
\sum_{j=0}^{\ell}\binom{\ell}{j}b_j \neq 0$, for all $0 \leq \ell \leq \prod_{i=1}^s p_i^{\alpha_i+\lfloor \log_{p_i}(n) \rfloor}$. 

Conversely assume that $\sum_{j=0}^{\ell}\binom{\ell}{j}b_j \neq 0$ for all $0 \leq \ell \leq \prod_{i=1}^s p_i^{\alpha_i+\lfloor \log_{p_i}(n) \rfloor}$. Then
$$
T^{2\ell}(f)=\Pi_n^\ell(f),
$$ for all $0 \leq \ell \leq \prod_{i=1}^s p_i^{\alpha_i+\lfloor \log_{p_i}(n) \rfloor}$; in particular, by the same argument as in the proof of Proposition \ref{prop:periodze}, for $k=\prod_{i=1}^s p_i^{\alpha_i+\lfloor \log_{p_i}(n) \rfloor}$, we have 
$$
T^{2 k}(f)=\Pi_n^{k}(f)=f,
$$ and so $f$ is $T$-periodic. This completes the proof.
\end{proof}
The following result is established in \cite{BP23} when $R=\F_p$. As mentionned earlier, it is easy to see that it also holds for any finite ring $R$.
\begin{prop}\label{prop:preperiodic} 
Every polynomial in $R[x]$ is eventually $T$-periodic.
\end{prop}

\begin{proof}

Clearly, the zero polynomial generates the trivial cycle $\{0\}$. Let $f=\sum_{i=0}^nb_kx^k \in R[x] \setminus \{0\}$, with $b_n \neq 0$.
By Lemma \ref{lem:basiresT}, we have $\deg(T^m(f)) \leq \deg(f)+1$ for all $m \geq 0$. We may replace $f$ by $T^m(f)$ ($m \geq 0$) such that $T^m(f)$ is of minimal degree in the orbit of $f$. In particular, $f$ is odd. In this case, an easy induction shows that $T^{2i}(f)$ is odd and $T^{2i+1}(f)$ is even for all $i \geq 0$. Note that the constant term of $T^{2i}(f)=\Pi_n^i(f)$ is 
$$
\sum_{j=0}^i \binom{i}{j}b_j \neq 0,
$$ for all $i \geq 0$. By Lemma \ref{crit:periodicity}, $f$ is $T$-periodic. This completes the proof.
\end{proof}

Together with Proposition \ref{prop:zero_char} and Proposition \ref{prop:periodze}, Proposition \ref{prop:preperiodic} completes the proof of Theorem \ref{main:thm1} of the introduction. We note that in the case where $R$ is a prime field, the main result of \cite{BP23} establishes a quadratic bound (in $\deg(f)$) on the number of $T$-iterations needed for a polynomial $f$ to reach a $T$-cycle. The authors do not know whether a similar bound holds for a general ring $R$ of positive characteristic -- this remains an interesting question for further study. However, in the case where $R$ is a finite field, we remark that the main result of \cite{BP23} (\cite[Theorem 1.1]{BP23}) holds without any change in the proofs, including the mentioned quadratic bound. The corresponding result is then:
\begin{thm}\label{thm:main} Let $\F_q$ be a finite field of characteristic $p$.
\begin{enumerate}[i)]
\item For any $k \geq 1$, a polynomial $f$ with $\deg(f)< p^k$ and $f(0) \neq 0$ is $T$-periodic if and only if its coefficients vector is of the form $B_kv$, where $v$ is a vector in $\F_q^{p^k}$ whose entries are all non-zero, and the matrix $B_k \in \mathrm{M}_{p^k \times p^k}(\F_q)$ is given by $B_{i,j} = (-1)^{i+j}{i-1 \choose j-1}$ if $i \geq j$ and $B_{i,j} = 0$ otherwise.\label{eq:e5}
    \item Let $f \in \F_q[x]$ be a polynomial of degree $d$. Then $T^{pd(d+1)-d}(f)$ is $T$-periodic.\label{eq:e1}
    \item Every $T$-periodic odd polynomial of degree $d$ has period $2p^k$ with $k=1+\lfloor \log_p(d) \rfloor$.\label{eq:e2}
    \item There are $q-1$ cycles of length $2$.\label{eq:e3}
    \item  For $k \geq 1$, the number of $T$-cycles of length $2p^k$ is given by $$\frac{(q-1)^{p^{k}}-(q-1)^{p^{k-1}}}{p^k}.$$\label{eq:e4}
\end{enumerate}
\end{thm}
\begin{proof}[Idea of the proof]
\begin{enumerate}[i)]
\item Let $k \geq 1$. Denote the inverse of $B_k$ by $A_k$. For an odd polynomial $f=\sum_{i=0}^{p^k-1}a_ix^i$, by \cite[Lemma 2.17]{BP23}, we have
\begin{equation}\label{eq:linkeqakai}
(A_k \cdot (a_0,\dots,a_{p^k-1})^t)_{\ell+1}=\sum_{j=0}^\ell\binom{\ell}{j}a_j,
\end{equation}
for all $0 \leq \ell \leq p^k-1$. By \cite[Lemma 2.15]{BP23} (which holds over any finite field), $f$ is $T$-periodic if and only if all the elements in \eqref{eq:linkeqakai} are nonzero. In other words, the vector coefficient of $f$ is of the form $B_kv$, where $v$ is a vector in $\F_q^{p^k}$ whose entries are all non-zero. See \cite[Lemma 2.18]{BP23} for more details.
\item The proof is the same as the one of \cite[Corollary 2.22]{BP23}, which is by induction on $d$.
\item This follows from Proposition \ref{prop:periodze}.
\item By Lemma \ref{cyc:oddeve}, any $T$-cycle contains an odd polynomial. Note that an odd polynomial $f \in \mathbb{F}_q[x]$ is $T$-periodic of period $2$, if and only if $T^2(f)=f$, that is, $\frac{(x+1)f-f_0}{x}=f$, or equivalently $f=f_0$. Therefore the only $T$-cycles of length $2$ are $(a,ax)$ with $a \in \mathbb{F}_q^*$.
\item Let $k \geq 1$. From \eqref{eq:e5} and \eqref{eq:e2}, we deduce that there are $(q-1)^{p^k}$ (resp., $(q-1)^{p^{k-1}}$) odd polynomials of period dividing $2p^{k}$ (resp., $2p^{k-1}$). Using Lemma \ref{cyc:oddeve}, we see that the number of $T$-cycles of length $2p^k$ is given by $$\frac{(q-1)^{p^{k}}-(q-1)^{p^{k-1}}}{p^k}.$$ See also \cite[Proposition 2.20]{BP23}.
\end{enumerate}
\end{proof}

\section{The Collatz map in rings of formal power series}\label{sec:power}
In this section, we prove Theorem \ref{main:thm2}.
\subsection{Arbitrary rings} Let $R$ be any commutative ring. We consider the analogue of the Collatz map for $R[[x]]$ as in \cite{BP23}:
\begin{align*}
T \colon R[[x]]& \rightarrow R[[x]]\\
         f        & \mapsto
         \begin{cases}
         (x+1)f-f_0 & \textrm{if}\,\, f_0 \neq 0,\\
         \frac{f}{x} & \textrm{otherwise}.
         \end{cases}
\end{align*}
Our aim (Theorem \ref{main:thm2}) is, in fact, when $R$ is finite, to characterize eventually $T$-periodic series and also to count the number of $T$-periodic cycles of a given period.
The method we use here is inspired by \cite{Lag90}, where the Collatz map on $\Z_2$ -- the arithmetic analogue of $R[[x]]$ -- is investigated. 

Let $\mathcal{I}_R$ be the set of all polynomials $f \in R[x]$ with $f_0 \in R^\times$; this set is clearly a multiplicative subset of $R[x]$. Denote by $\mathcal{S}_R$ the localization of $R[x]$ at $\mathcal{I}_R$. Since series in $\mathcal{I}_R$ are invertible in $R[[x]]$, the morphism $i_R: f/g \in \mathcal{S}_{R} \hookrightarrow f g^{-1} \in R[[x]]$ ($f \in R[x]$ and $g \in \mathcal{I}_R$) is injective. In the sequel, we may identify $\mathcal{S}_R$ with its image in $R[[x]]$. Explicitly, we have
$$
\mathcal{S}_R=\left\lbrace u(1+xv)^{-1} \mid u,v \in R[x] \right\rbrace.
$$
In our investigation, we shall use the following modified Collatz map:
\begin{align*}
\overline{T} \colon R[[x]]& \rightarrow R[[x]]\\
         f        & \mapsto
         \begin{cases}
         \frac{(x+1)f-f_0}{x} & \textrm{if}\,\, f_0 \neq 0,\\
         \frac{f}{x} & \textrm{otherwise}.
         \end{cases}
\end{align*}
For any $v \in R$, let
$$
p(v)=
\begin{cases}
 0 & \textrm{if}\,\,v=0,\\
 1 & \textrm{if}\,\, v \neq 0,
\end{cases}
$$
and let 
\begin{align*}
\overline{T}_v \colon R((x))& \rightarrow R((x))\\
         f        & \mapsto
         \frac{(x+1)^{p(v)}f-v}{x}.
\end{align*}
For any $f \in R[[x]]$, we define
$${\bf v}^T_{f}=\left({\bf v}^T_f(0),{\bf v}_f^T(1),{\bf v}^T_f(2),\dots\right)=\left(T^0(f)(0),T^1(f)(0),T^2(f)(0),\dots \right) \in R^{\mathbb{N}}$$ (resp., $${\bf v}^{\T}_f=({\bf v}^{\T}_f(0),{\bf v}^{\T}_f(1),{\bf v}^{\T}_f(2),\dots)=\left(\overline{T}^0(f)(0),\overline{T}^1(f)(0),\overline{T}^2(f)(0),\dots\right) \in R^{\mathbb{N}}.)$$ Also, for $n \geq 1$, let ${\bf v}^T_{f,n}=\left({\bf v}^T_f(0),\dots, {\bf v}^T_f(n-1)\right) \in R^n$ and ${\bf v}^{\T}_{f,n}=\left({\bf v}^{\T}_f(0),\dots, {\bf v}^{\T}_f(n-1)\right) \in R^n$. Moreover, for $n \geq 1$, let 
$$
\Omega_n =\left\lbrace f \in R[[x]] \mid T^n(f)=f \right \rbrace,
$$
$$
\overline{\Omega}_n =\left\lbrace f \in R[[x]] \mid \overline{T}^n(f)=f \right \rbrace,
$$
and
$$
\Omega^*_n= \left\lbrace f \in R[[x]] \mid f\,\, \textrm{is $T$-periodic of $T$-period}\,\,n  \right\rbrace.
$$
For ${\bf v}=(v_0,v_1,\dots,v_{n-1}) \in R^n$ ($n\geq 1$), we define the map
$$
\T_{{\bf v}}=\overline{T}_{v_{n-1}}\circ\cdots \circ \overline{T}_{v_0}:R((x)) \rightarrow R((x)).$$
\begin{lem}\label{lem:fundiden}
For ${\bf v}=(v_0,\dots,v_{n-1}) \in R^n$ ($n \geq 1$) and $f \in R[[x]]$, we have
$$
\T_{{\bf v}}(f)=\frac{1}{x^{n}}\left( (x+1)^{\sum_{k=0}^{n-1}p(v_k)}f-\sum_{j=0}^{n-1}v_jx^{j}(x+1)^{\sum_{k=j+1}^{n-1}p(v_k)} \right).
$$
\end{lem}
\begin{proof}
Let $f \in R[[x]]$. We prove the claim by induction on the length $n$ of ${\bf v}$. This is clear when $n=1$. Assume this is true for some $n \geq 1$. Let ${\bf v}=(v_0,\dots,v_n) \in R^{n+1}$ of length $n+1$. Write ${\bf v}'=(v_0,\dots,v_{n-1}) \in R^n$. By the induction assumption, we have
$$
\T_{{\bf v}'}(f)=\frac{1}{x^{n}}\left( (x+1)^{\sum_{k=0}^{n-1}p(v_k)}f-\sum_{j=0}^{n-1}v_jx^{j}(x+1)^{\sum_{k=j+1}^{n-1}p(v_k)} \right).
$$ Hence we obtain
\begin{align*}
\T_{{\bf v}}(f)&=\T_{v_n}(\T_{{\bf v}'}(f))\\
&=\frac{(x+1)^{p(v_n)}\T_{{\bf v}'}(f)-v_n}{x}\\
&=\frac{(x+1)^{p(v_n)}\left( \frac{1}{x^{n}}\left( (x+1)^{\sum_{k=0}^{n-1}p(v_k)}f-\sum_{j=0}^{n-1}v_jx^{j}(x+1)^{\sum_{k=j+1}^{n-1}p(v_k)} \right) \right)-v_n}{x}\\
&=\frac{1}{x^{n+1}}\left( (x+1)^{\sum_{k=0}^{n}p(v_k)}f-\sum_{j=0}^{n}v_jx^{j}(x+1)^{\sum_{k=j+1}^{n}p(v_k)} \right).
\end{align*}
Therefore the claim holds for $n+1$. This completes the proof.
\end{proof}
Roughly speaking, the following result asserts that any $\overline{T}$-periodic series $f$ is uniquely determined by ${\bf v}_f^{\T}$:
\begin{prop}\label{prop:prescribconsterm}
Let ${\bf v}=(v_0,v_1,\dots,v_{n-1}) \in R^n$ ($n \geq 1$). Then there exists a unique  $f \in \overline{\Omega}_n$ such that ${\bf v}^{\T}_{f,n}={\bf v}$, given by:
$$
f=\frac{\sum_{j=0}^{n-1}v_jx^{j}(x+1)^{\sum_{k=j+1}^{n-1}p(v_{k})}}{(x+1)^{\sum_{k=0}^{n-1}p(v_k)}-x^{n}}.
$$
\end{prop}
\begin{proof}
Let $f \in \overline{\Omega}_n$ such that ${\bf v}^{\T}_{f,n}={\bf v}$. Then, we have 
\begin{equation*}\label{eq:equcom}
f=\T^n(f)=\overline{T}_{{\bf v}}(f),
\end{equation*} and, by Lemma \ref{lem:fundiden}, we get
$$
f=\frac{1}{x^{n}}\left( (x+1)^{\sum_{k=0}^{n-1}p(v_k)}f-\sum_{j=0}^{n-1}v_jx^{j}(x+1)^{\sum_{k=j+1}^{n-1}p(v_{k})} \right),
$$
from which we obtain 
\begin{equation*}\label{eq:fformu}
f=\frac{\sum_{j=0}^{n-1}v_jx^{j}(x+1)^{\sum_{k=j+1}^{n-1}p(v_{k})}}{(x+1)^{\sum_{k=0}^{n-1}p(v_k)}-x^{n}} \in \mathcal{S}_R.
\end{equation*}
Hence we have the unicity. 

Now we are going to prove that, for 
\begin{equation}\label{eq:exprg}
g=\frac{\sum_{j=0}^{n-1}v_jx^{j}(x+1)^{\sum_{k=j+1}^{n-1}p(v_{k})}}{(x+1)^{\sum_{k=0}^{n-1}p(v_k)}-x^{n}} \in \mathcal{S}_R,
\end{equation}
 we have $\overline{T}^n(g)=g$ (that is, $g \in \overline{\Omega}_n$) and that ${\bf v}^{\T}_{g,n}={\bf v}$. Since $g=\overline{T}_{{\bf v}}(g)$ by the above argument, it suffices to prove that ${\bf v}^{\T}_{g,n}={\bf v}$. Let us prove, by induction, that ${\bf v}^{\T}_g(\ell)=v_\ell$, for all $0 \leq \ell \leq n-1$. For $j=0$: it is clear that $g(0)=v_0$, and so ${\bf v}^{\T}_g(0)=v_0$. Assume that ${\bf v}^{\T}_g(0)= v_0, \dots,{\bf v}^{\T}_g(\ell)= v_\ell$, for some $ 0 \leq \ell< n-1$. Then 
\begin{align*}
\overline{T}^{\ell+1}(g)&=\overline{T}_{(v_0,\dots,v_\ell)}(g)\\
&=\frac{1}{x^{\ell+1}}\left( (x+1)^{\sum_{k=0}^{\ell}p(v_k)}g-\sum_{j=0}^{\ell}v_jx^{j}(x+1)^{\sum_{k=j+1}^{\ell}p(v_k)} \right)\quad \textrm{by Lemma \ref{lem:fundiden}}\\
&=\frac{1}{(x+1)^{\sum_{k=0}^{n-1}p(v_k)}-x^{n}} \left((x+1)^{\sum_{k=0}^\ell p(v_k)} \left( v_{\ell+1}+xA(x) \right)+ x^{n-\ell-1} B(x) \right),
\end{align*} for some $A(x),B(x) \in R[x]$. Noticing that $n-\ell-1 \geq 1$, we have $\left[\overline{T}^{\ell+1}(g)\right](0)=v_{\ell+1}$, and so ${\bf v}^{\T}_{g}(\ell+1)=v_{\ell+1}$. This completes the proof.
\end{proof}
As a generalization of Proposition \ref{prop:prescribconsterm}, one may ask the following; note that it holds true for the condensed Collatz map on $\mathbb{Z}_2$.
\begin{qst}
For any ${\bf v} \in R^{\mathbb{N}}$, is there $f \in R[[x]]$ such that ${\bf v}^{\overline{T}}_f={\bf v}$?
\end{qst}
\begin{cor}\label{cor:inclbaromega}
For all $n \geq 1$, we have $\overline{\Omega}_n \subset \mathcal{S}_R$.
\end{cor}
\begin{proof}
Let $f \in \overline{\Omega}_n$. Write ${\bf v}^{\overline{T}}_{f,n}=(v_0,\dots,v_{n-1})$. By Proposition \ref{prop:prescribconsterm}, we have
$$
f=\frac{\sum_{j=0}^{n-1}v_jx^{j}(x+1)^{\sum_{k=j+1}^{n-1}p(v_{k})}}{(x+1)^{\sum_{k=0}^{n-1}p(v_k)}-x^{n}},
$$ which clearly belongs to $\mathcal{S}_R$. Therefore we have $\overline{\Omega}_n \subset \mathcal{S}_R$.
\end{proof}
\begin{defn}
A vector $(v_0,v_1,\dots) \in R^{\mathbb{N}}$ is called {\it zero dense} if it contains no consecutive nonzero coordinates; in other words, for all $i \geq 0$, we have $v_i=0$ or $v_{i+1}=0$. A vector $(v_0,v_1,\dots,v_{n-1}) \in R^n$ ($n \geq 1$) is called {\it cyclically zero dense} if the infinite concatenation 
$$
(v_0,v_1,\dots,v_{n-1}, v_0,v_1,\dots,v_{n-1}, \dots, v_0,v_1,\dots,v_{n-1},\dots) \in R^{\mathbb{N}}
$$ is zero dense in the preceding sense.
\end{defn}
The following observation follows directly from the definition of $T$ and $\overline{T}$:
\begin{lem}\label{lem:essnpresc}
Let $f \in R[[x]]$. Then ${\bf v}^T_{f}$ is obtained from ${\bf v}^{\T}_f$ by inserting $0$ after each nonzero coordinate.   Conversely ${\bf v}^{\T}_f$ is obtained from ${\bf v}^T_{f}$ by removing the zero coordinate following any nonzero coordinate. In particular, ${\bf v}^T_f$ is zero dense.  
\end{lem}
Moreover we have the following:
\begin{lem}\label{lem:backforth}
Let $f \in R[[x]]$.
\begin{enumerate}
    \item Let $n \geq 1$ and let $s$ be the number of nonzero coordinates in ${\bf v}^{\T}_{f,n}$. Then we have $\overline{T}^n(f)=T^{n+s}(f)$. Moreover, ${\bf v}^T_{f,n+s}$ is obtained from ${\bf v}^{\T}_{f,n}$ by inserting $0$ after each nonzero coordinate.
    \item Let $n \geq 1$. Assume that ${\bf v}^T_f(n-1)=0$ and let $s$ denote the number of nonzero coordinates in ${\bf v}^T_{f,n}$. Then we have $T^n(f)=\overline{T}^{n-s}(f)$. Moreover, ${\bf v}^{T}_{f,n-s}$ is obtained from ${\bf v}^T_{f,n}$ by removing the zero coordinate following any nonzero coordinate. 
\end{enumerate}
\end{lem}
The following result, which follows immediately from Proposition \ref{prop:prescribconsterm} and Lemma \ref{lem:backforth}, is needed to count $T$-cycles in \S\ref{subs:finirin}. 
\begin{cor}\label{prop:cycfree}
For each cyclically zero dense vector ${\bf v}=(v_0,v_1,\dots,v_{n-1}) \in R^n$ ($n \geq 2$), there exists a unique $f \in \Omega_n$ such that ${\bf v}^T_{f,n}={\bf v}$.
\end{cor}
The following follows from the proof of \cite[Lemma 3.6]{BP23}.
\begin{lem}\label{lem:fromlem36}
Let $f=u(1+xv)^{-1} \in \mathcal{S}_R$ with $u,v \in R[x]$. Then $T^{k}(f)=u_k(1+x v)^{-1}$ ($k \geq 0$ and $u_k \in R[x]$) where $\deg(u_k)$ is bounded as $k$ increases.
\end{lem}
\subsection{Finite rings}\label{subs:finirin} Assume that $R$ is {\bf finite} of cardinality $q$. We obtain the first part of Theorem \ref{main:thm2}:
\begin{thm}
The eventually $T$-periodic series are exactly the series in $\mathcal{S}_R$.
\end{thm}
\begin{proof}
Let $0 \neq f$ be an eventually $T$-periodic series. Note that $T(f) \in \mathcal{S}_R$ implies $f \in \mathcal{S}_R$. Hence we may assume that $f$ is $T$-periodic. Then there exists $n \geq 1$ such that $T^n(f)=f$. Note that $T^n(T(f))=T(f)$ and, if ${\bf v}^T_{f}(n-1)\neq 0$ then ${\bf v}^T_{T(f)}(n-1)={\bf v}^T_{f}(n)=0$. Hence, up to replacing $f$ by $T(f)$, we may assume ${\bf v}^T_{f}(n-1)=0$. In this case, by Lemma \ref{lem:backforth}, there exists $0 \leq s \leq n $ such that $f=T^n(f)=\overline{T}^{n-s}(f)$, and so $f \in \overline{\Omega}_{n-s}$. By Corollary \ref{cor:inclbaromega}, we get $f \in \mathcal{S}_R$.

Conversely, let $f \in \mathcal{S}_R$. Since $R$ is finite, by Lemma \ref{lem:fromlem36}, the orbit of $f$ is finite, that is, $f$ is eventually $T$-periodic.
\end{proof}

For $n \geq 1$, by Lemma \ref{lem:essnpresc} and Corollary \ref{prop:cycfree}, we can identify $\Omega_n$ with the set $\mathcal{F}_n$ of all cyclically zero dense vectors of $R^n$, and so $\Omega_n$ is finite. We shall now count the number of $T$-cycles of a given length. For $n \geq 1$, let $i_n=|\Omega^*(n)|$ and $j_n=|\Omega(n)|$. For $n \geq 1$, we have the disjoint union $\Omega(n)=\bigsqcup_{d|n} \Omega^*(d)$, and so we obtain $j_n=\sum_{d|n}i_n$. Hence, for all $n \geq 1$, the M\"{o}bius inversion formula gives
\begin{equation}\label{eq:formIn}
i_n=\sum_{d|n}\mu(d)j_{\frac{n}{d}},
\end{equation} where $\mu$ denotes the M\"{o}bius function.

Let $n \geq 2$. Let $E_n$ be the set of all zero dense vectors in $R^n$ and let $e_n=|E_n|$. Notice that $\mathcal{F}_n \subset E_n$. For $(v_0,\dots,v_{n-1}) \in R^n$, we have $(v_0,\dots,v_{n-1}) \in \mathcal{F}_n$ if and only if $v_{n-1}=0$ and $(v_0,v_{1},\dots,v_{n-2}) \in E_{n-1}$, or if $v_{n-1} \neq 0$, $v_{0}=v_{n-2}=0$ and $(v_{1},\dots,v_{n-3}) \in E_{n-3}$. Hence for $n \geq 3$, we have 
\begin{equation}\label{eq:recformjnen}
 j_n=|\mathcal{F}_n|=e_{n-1}+(q-1)e_{n-3}.   
\end{equation} Moreover we have $j_1=1$ (because $\Omega_1=\{0\}$) and $j_2=2q-1$ (because $\mathcal{F}_2=\{(0,0)\} \cup \{(0,a),(a,0) \mid a \neq 0\}$). We now compute $e_n=|E(n)|$.
\begin{lem}\label{lem:formjn}
Let $b=4q-3$, $\alpha=\frac{1+\sqrt{b}}{2}$ and $\beta=\frac{1-\sqrt{b}}{2}$. Then
\begin{equation}\label{eq:explformome}
e_n=\frac{\alpha^{n+2}-\beta^{n+2}}{\sqrt{b}},
\end{equation} for all $n \geq 1$.
\end{lem}
\begin{proof}
Consider the following Fibonacci-like sequence:
\begin{equation}\label{eq:recomeg}
\begin{cases}
f_0=1,\\
f_1=q,\\
f_n=f_{n-1}+(q-1)f_{n-2},\,\, n \geq 2.
\end{cases}
\end{equation}
Note that the characteristic polynomial of the linear recurrence \eqref{eq:recomeg} is $X^2-X-(q-1)$. Since its roots are $\alpha$ and $\beta$, we have
$$
f_n=A\alpha^n+B\beta^n\,\,(n \geq 0),
$$ for some $A,B \in \mathbb{C}$. Using the initial conditions 
\begin{align*}
A+B=f_0=1\\ 
A\alpha + B \beta =f_1=q,
\end{align*} we get
$A=\dfrac{q-\beta}{\sqrt{b}}$ and $B=\dfrac{\alpha-q}{\sqrt{b}}$. Therefore we have
\begin{equation}
f_n=\frac{(q-\beta)\alpha^n-(q-\alpha)\beta^n}{\sqrt{b}}\,\,(n \geq 0).
\end{equation}
Since $\alpha^2=q-\beta$ and $\beta^2=q-\alpha$, we obtain
\begin{equation*}
f_n=\frac{\alpha^{n+2}-\beta^{n+2}}{\sqrt{b}}\,\,(n \geq 0).
\end{equation*} 

To complete the proof of the lemma, it remains to prove that $(e_n)_{n \geq 1}$ also satisfies \eqref{eq:recomeg}.
For $n=1,2$, it follows from $E_1=R$ and $E_2=\{(0,0)\} \cup \{(0,a), (a,0) \mid a \in R \setminus\{0\} \}$. Let $n \geq 2$. Observe that $(v_0,\dots,v_{n-1}) \in E_n$ if and only if ($v_{0}=0$ and $(v_1,\dots,v_{n-1}) \in \Omega_{n-1}$) or ($v_{0} \neq 0$, $v_1=0$ and $(v_{2},\dots,v_{n-1}) \in \Omega_{n-2}$). This shows that $(e_n)_{n \geq 1}$ satisfies \eqref{eq:recomeg}.
\end{proof}

The following is the second part of Theorem \ref{main:thm2}: 
\begin{prop}\label{prop:resprin}
With the same notation as in Lemma \ref{lem:formjn}, for any $n \geq 3$, the number of $T$-cycles of length $n$ is:
\begin{equation}\label{eq:Zn53}
Z_n=\frac{1}{n}\sum_{d|n}\mu(d)\left(\alpha^{\frac{n}{d}}+\beta^{\frac{n}{d}}\right)
\end{equation}
\end{prop}
\begin{proof}
Combining \eqref{eq:formIn} and Equation \eqref{eq:recformjnen}, for $n \geq 3$, we have
$$
i_n=\sum_{d \mid n}\mu(d)j_{\frac{n}{d}}=\sum_{d \mid n}\mu(d)\left(e_{\frac{n}{d}-1}+(q-1)e_{\frac{n}{d}-3}\right)
$$
A straightforward computation using Lemma \ref{lem:formjn}, $\alpha^2+q-1=\alpha\sqrt{b}$ and $\beta^2+q-1=-\beta\sqrt{b}$ leads to: 
$$
Z_n=\frac{i_n}{n}=\frac{1}{n}\sum_{d|n}\mu(d)\left(\alpha^{\frac{n}{d}}+\beta^{\frac{n}{d}}\right).
$$
\end{proof}
\begin{cor}\label{cor:asym}
With the same notation as in Lemma \ref{lem:formjn}, we have
$$
Z_n \sim \frac{\alpha^n}{n},
$$ as $n \rightarrow  \infty$.
\end{cor}
\section{The Collatz map in $\mathbb{Z}_2$}\label{sec:z2}
In this section, we prove Theorem \ref{main_z2}. As mentioned in the introduction, Lagarias proved in \cite{Lag90} that, under the ``condensed" Collatz map:
\begin{align*}
\overline{T} \colon \mathbb{Z}_2& \rightarrow \mathbb{Z}_2\\
         f        & \mapsto
         \begin{cases}
         \frac{f}{2} & \textrm{if}\,\, f \equiv 0 \pmod{2},\\
         \frac{3f+1}{2} & \textrm{otherwise},
         \end{cases}
\end{align*}
the number of cycles of length $n \geq 1$ is 
$$
I(n)=\frac{1}{n}\sum_{d|n} \mu(d)2^{\frac{n}{d}}.
$$
In this section, using the same approach as in \S\ref{sec:power}, we determine the number of cycles of length $n \geq 1$ for the classical Collatz map:
\begin{align*}
T \colon \mathbb{Z}_2& \rightarrow \mathbb{Z}_2\\
         f        & \mapsto
         \begin{cases}
         \frac{f}{2} & \textrm{if}\,\, f \equiv 0 \pmod{2},\\
          3f+1 & \textrm{otherwise}.
         \end{cases}
\end{align*}
We shall use similar notation to that of \S\ref{sec:power}: For any $v \in \{0,1\}$, let 
\begin{align}
\overline{T}_v \colon \mathbb{Q}_2& \rightarrow \mathbb{Q}_2\\
         f        & \mapsto
         \frac{3^{v}f+v}{2}.
\end{align}
(where $\mathbb{Q}_2$ denotes the field of $2$-adic numbers -- the quotient field of $\Z_2$). For any $f \in \mathbb{Z}_2$, let $f_0 \in \{0,1\}$ be such that $f \equiv f_0 \pmod{2}$. For any $f \in \mathbb{Z}_2$, we define
$${\bf v}^T_{f}=\left({\bf v}^T_f(0),{\bf v}_f^T(1),{\bf v}^T_f(2),\dots\right)=\left(T^0(f)_0,T^1(f)_0,T^2(f)_0,\dots \right) \in \{0,1\}^{\mathbb{N}}$$ (resp., $${\bf v}^{\T}_f=({\bf v}^{\T}_f(0),{\bf v}^{\T}_f(1),{\bf v}^{\T}_f(2),\dots)=\left(\overline{T}^0(f)_0,\overline{T}^1(f)_0,\overline{T}^2(f)_0,\dots\right) \in \{0,1\}^{\mathbb{N}}.)$$ Also, for $n \geq 1$, let ${\bf v}^T_{f,n}=\left({\bf v}^T_f(0),\dots, {\bf v}^T_f(n-1)\right) \in \{0,1\}^n$ and ${\bf v}^{\T}_{f,n}=\left({\bf v}^{\T}_f(0),\dots, {\bf v}^{\T}_f(n-1)\right) \in \{0,1\}^n$. For $n \geq 1$, let 
$$
\Omega_n =\left\lbrace f \in \mathbb{Z}_2 \mid T^n(f)=f \right \rbrace,
$$
$$
\overline{\Omega}_n =\left\lbrace f \in \mathbb{Z}_2 \mid \overline{T}^n(f)=f \right \rbrace,
$$
and
$$
\Omega^*_n= \left\lbrace f \in \mathbb{Z}_2 \mid f\,\, \textrm{is $T$-periodic of $T$-period}\,\,n  \right\rbrace.
$$
The following result, an analogue of Proposition \ref{prop:prescribconsterm}, was proved in \cite[Theorem 2.1.]{Lag90}.
\begin{prop}\label{prop:z2period}
Let ${\bf v}=(v_0,v_1,\dots,v_{n-1}) \in \{0,1\}^n$ ($n \geq 1$). Then there exists a unique  $f \in \overline{\Omega}_n$ such that ${\bf v}^{\T}_{f,n}={\bf v}$, given by:
$$
f=\frac{\sum_{j=0}^{n-1}v_j2^{j}3^{\sum_{k=j+1}^{n-1}v_{k}}}{2^{n}-3^{\sum_{k=0}^{n-1}v_k}}.
$$
\end{prop}
\begin{rmk}
Regarding Proposition \ref{prop:z2period}, although the statement of \cite[Theorem 2.1.]{Lag90} concerns $\mathbb{Q} \cap \mathbb{Z}_2$, we may replace $\mathbb{Q} \cap \mathbb{Z}_2$ by $\mathbb{Z}_{2}$ without any change in the proof. (We note that there is a minor difference in terminology between here and \cite{Lag90} -- there an element $f$ is called ``periodic of period $n$" if $\overline{T}^n(f) = f$, and $n$ need not be minimal with this property.)
\end{rmk}
Concerning eventually $\overline{T}$-periodic elements of $\mathbb{Z}_2$, it is known that they are all rational (see \cite{Aki04,Roz19}). Lagarias conjectured that the converse also holds (see \cite{Lagarias2010}).
\begin{conj}[Periodicity Conjecture, Lagarias]
The eventually $\overline{T}$-periodic elements of $\mathbb{Z}_2$ are exactly the elements of $\mathbb{Q} \cap \mathbb{Z}_2$.
\end{conj}

\begin{defn}
A vector $(v_0,v_1,\dots) \in \{0,1\}^{\mathbb{N}}$ is called {\it zero dense} if it contains no consecutive $1$'s; in other words, for all $i \geq 0$, we have $v_i=0$ or $v_{i+1}=0$. A vector $(v_0,v_1,\dots,v_{n-1}) \in \{0,1\}^n$ ($n \geq 1$) is called {\it cyclically zero dense} if the infinite concatenation 
$$
(v_0,v_1,\dots,v_{n-1}, v_0,v_1,\dots,v_{n-1}, \dots, v_0,v_1,\dots,v_{n-1},\dots) \in \{0,1\}^{\mathbb{N}}
$$ is zero dense in the preceding sense.
\end{defn}
\begin{rmk}
In symbolic dynamics, the set of zero dense sequences is called the \emph{golden mean} shift space (see \cite{LM95}). 
\end{rmk}
The proofs of Lemma \ref{lem:denZ2}, Lemma \ref{lem:fromUtoTZ2} and Proposition \ref{prop:cycfreeZ2} below are exactly the same as the proofs of Lemma \ref{lem:essnpresc}, Lemma \ref{lem:backforth} and Corollary \ref{prop:cycfree} respectively.
\begin{lem}\label{lem:denZ2}
Let $f \in \mathbb{Z}_2$. Then ${\bf v}^T_{f}$ is obtained from ${\bf v}^{\T}_f$ by inserting $0$ after each $1$. Conversely ${\bf v}^{\T}_f$ is obtained from ${\bf v}^T_{f}$ by removing the zero coordinate following any $1$. In particular, ${\bf v}^T_f$ is zero dense. 
\end{lem}

\begin{lem}\label{lem:fromUtoTZ2}
Let $f \in \mathbb{Z}_2$.
\begin{enumerate}
    \item Let $n \geq 1$ and let $s$ be the number of $1$ in ${\bf v}^{\T}_{f,n}$. Then we have $\overline{T}^n(f)=T^{n+s}(f)$. Moreover, ${\bf v}^T_{f,n+s}$ is obtained from ${\bf v}^{\T}_{f,n}$ by inserting $0$ after each $1$.
    \item Let $n \geq 1$. Assume that ${\bf v}^T_f(n-1)=0$ and let $s$ denote the number of $1$ in ${\bf v}^T_{f,n}$. Then we have $T^n(f)=\overline{T}^{n-s}(f)$. Moreover, ${\bf v}^{T}_{f,n-s}$ is obtained from ${\bf v}^T_{f,n}$ by removing the zero coordinate following any $1$. 
\end{enumerate}
\end{lem}

\begin{prop}\label{prop:cycfreeZ2}
For each cyclically zero dense vector ${\bf v}=(v_0,v_1,\dots,v_{n-1}) \in \{0,1\}^n$ ($n \geq 2$), there exists a unique $f \in \Omega_n$ such that ${\bf v}^T_{f,n}={\bf v}$.
\end{prop}
With Lemma \ref{lem:denZ2}, Lemma \ref{lem:fromUtoTZ2} and Proposition \ref{prop:cycfreeZ2} in place of Lemma \ref{lem:essnpresc}, Lemma \ref{lem:backforth} and Corollary \ref{prop:cycfree} respectively, all the results in \S\ref{subs:finirin} hold true for $\mathbb{Z}_2$ without any change in the proofs by replacing $q$ with $2$. We thus get Theorem \ref{main_z2} which we recall:
\vskip 1mm
\noindent
{\bf Theorem \ref{main_z2}.} Let $n \geq 1$. Then the number $Z_n$ of $T$-cycles of length $n$ in $\mathbb{Z}_2$ is finite. Moreover, we have 
\begin{align}\label{ZnforZ2}
Z_n=\frac{1}{n}\sum_{d|n}\mu(d) \left( \phi^{\frac{n}{d}}+\psi^{\frac{n}{d}}\right),
\end{align}
where $\phi=\frac{1+\sqrt{5}}{2}$ and $\psi=\frac{1-\sqrt{5}}{2}$.
Asymptotically, we have
\begin{equation}\label{eq:asymZ2}
Z_n \sim \frac{\phi^n}{n},
\end{equation} as $n \rightarrow  \infty$.
\vskip 1mm
\noindent

To illustrate Theorem \ref{main_z2}, we computed the first terms:
\begin{itemize}
\item $Z_1=1$ which corresponds to the fixed point $0$,
\item $Z_2=1$ which corresponds to the cycle $(-1,-2)$,
\item $Z_3=1$ which corresponds to the cycle $(1,4,2)$,
\item $Z_4=1$ which corresponds to the cycle $\left(\frac{1}{5},\frac{8}{5},\frac{4}{5},\frac{2}{5}\right)$,
\item $Z_5=2$ which corresponds to the cycles $(-10,-5,-14,-7,-20)$ and $\left(\frac{8}{13},\frac{4}{13},\frac{2}{13},\frac{1}{13},\frac{16}{13} \right)$.
\end{itemize}
\begin{rmk}
The sequence $(Z_n)_{n \geq 1}$ is the same as the sequence A006206 in OEIS \cite{Oeis2024}.
\end{rmk}

\section*{Acknowledgements}
The first author is grateful for the support of a Technion fellowship, of an Open University of Israel post-doctoral fellowship, and of the Israel Science Foundation (grant no. 353/21). The authors would like to thank the referees for their careful reading and their suggestions that helped improve the presentation of the paper.

\end{document}